\documentclass{amsart}
\usepackage{amssymb, latexsym}

\newcommand{\m}{\mu}

\newcounter{cnt1}
\newcounter{cnt2}
\newcounter{cnt3}
\newcommand{\blr}{\begin{list}{$($\roman{cnt1}$)$}
 {\usecounter{cnt1} \setlength{\topsep}{0pt}
 \setlength{\itemsep}{0pt}}}
\newcommand{\bla}{\begin{list}{$($\alph{cnt2}$)$}
 {\usecounter{cnt2} \setlength{\topsep}{0pt}
 \setlength{\itemsep}{0pt}}}
\newcommand{\bln}{\begin{list}{$($\arabic{cnt3}$)$}
 {\usecounter{cnt3} \setlength{\topsep}{0pt}
 \setlength{\itemsep}{0pt}}}
\newcommand{\el}{\end{list}}
\newtheorem{thm}{Theorem}[section]

\newtheorem{Def}[thm]{Definition}
\newtheorem{prop}[thm]{Proposition}
\newtheorem{rem}[thm]{Remark}
\newcommand{\Rem}{\begin{rem} \rm}
\newcommand{\bdfn}{\begin{Def} \rm}
\newcommand{\edfn}{\end{Def}}

\newcommand{\ba}{\begin{array}}
\newcommand{\ea}{\end{array}}

\begin{document} 
	\begin{center} 
{\Large{\bf{ Kuelbs-Steadman spaces for 
Banach space-valued measures}}}
	\normalsize
	\vspace{0.3cm} 
	
	Antonio Boccuto$^{1,\ast}$, Bipan Hazarika$^2$ and Hemanta Kalita$^3$
		\vspace{0.3cm}
		
\noindent $^1$Department of Mathematics and Computer Sciences, via Vanvitelli, 1  I-06123 Perugia, Italy\\	
$^2$Department of  Mathematics, Gauhati University, Guwahati-781014, Assam, India.\\
	$^3$Department of Mathematics, Patkai Christian College (Autonomous),
	Dimapur, Patkai-797103, Nagaland, India	\\

\noindent E-mail:   antonio.boccuto@unipg.it; 
bh\_rgu@yahoo.co.in; bh\_gu@gauhati.ac.in;\\
hemanta30kalita@gmail.com

\normalsize
\title{}

\thanks{{\today; $^\ast$Corresponding author}}
\maketitle
\begin{abstract}
We introduce Kuelbs-Steadman-type spaces for 
real-valued functions,
with respect to countably additive measures, taking values in
Banach spaces. We investigate their main properties
and embeddings in $L^p$-type spaces, considering 
both the norm associated to norm convergence of the 
involved integrals and that related to weak convergence   
of the integrals.
\noindent {\bf AMS Subject Classification No:} 26A39, 46B25, 46E35, 46E39, 46F25. 
\end{abstract}
\end{center}
\section{Introduction}
Kuelbs-Steadman spaces have been the subject of many
recent studies (see e.g. \cite{GILLSURVEY, GZ, KHM} 
and the references therein). The investigation
of such spaces arises from the idea to consider the 
$L^1$ spaces as embedded in a larger Hilbert space with
smaller norm, and containing in a certain sense the 
Henstock-Kurzweil integrable functions.  This allows 
to give several applications to Functional Analysis and
other branches of Mathematics, for instance Gaussian 
measures (see also \cite{KUELBS}), convolution operators,
Fourier transforms, Feynman integral, quantum mechanics,
differential equations and Markov 
chains (see also \cite{GILLSURVEY, GZ, KHM}).
This approach allows also to develop a theory of Functional 
Analysis which includes Sobolev-type spaces, 
in connection with Kuelbs-Steadman spaces rather
than with classical $L^p$ spaces. 

In this paper we extend the theory of 
Kuelbs-Steadman spaces to measures $\mu$
defined on a $\sigma$-algebra
and with values in a Banach space $X$.
We consider an integral 
for real-valued functions $f$ with respect to
$X$-valued countably additive measures.
In this setting,
a fundamental role is played by the separability of $\mu$. 
This condition is satisfied, for instance,
when $T$ is a metrizable separable space, not 
necessarily with a Schauder basis (such spaces exist,
see for instance \cite{GILLSURVEY}), and $\mu$ is a 
Radon measure. In the literature, some 
deeply investigated particular cases are when 
$X=\mathbb{R}^n$ and $\mu$ is the Lebesgue measure,
and when $X$ is a Banach space with 
a Schauder basis (see also \cite{GILLSURVEY, GZ, KHM}).
Since the integral of $f$ with respect to $\mu$
is an element of $X$, in general it is  
not natural to define an inner product, when it is dealt
with norm convergence of the involved integrals. Moreover, 
when $\mu$ is a vector measure, the spaces 
$L^p[\mu]$ do not satisfy all classical properties as 
the spaces $L^p$ with respect to a scalar measure
(see also \cite{IRENE, OKADA, EP}).
However, it is always possible 
to define Kuelbs-Steadman spaces as Banach 
spaces, which are completions of suitable $L^p$ spaces. 
We introduce them and prove that they are normed 
spaces, and 
that the embeddings of $KS^p[\mu]$ into
$L^q[\mu]$ are continuous and dense. Moreover, we show that the 
norm of $KS^p$ spaces is smaller than that related to the 
space of all Henstock-Kurzweil integrable functions 
(the Alexiewicz norm). Furthermore, we prove that 
$KS^p$ spaces are K\"{o}the function spaces and
Banach lattices, extending to the setting of
$K{S^p}[\mu]$-spaces some results proved in 
\cite{EP} for spaces of type $L^p[\mu]$.
Furthermore, when $X^{\prime}$ is separable, 
it is possible to consider a topology associated to 
weak convergence of integrals and to define 
a corresponding norm and an inner product.
We introduce the Kuelbs-Steadman spaces related to
this norm, and prove the analogous properties
investigated for $KS^p$ spaces related to
norm convergence of the integrals.
In this case, 
since we deal with a separable Hilbert space, it is possible to 
consider operators like convolution and Fourier transform, and 
to extend the theory developed
in \cite{GILLSURVEY, GZ, KHM}
to the context of Banach space-valued measures.  
\section{Vector measures, $(HKL)$- and $(KL)$-integrals}
Let $T \neq \emptyset$ be an abstract set,
${\mathcal P}(T)$ be the class of all subsets of 
$T$, $\Sigma \subset {\mathcal P}(T)$ be
a $\sigma$-algebra, 
$X$ be a Banach space and $X^{\prime}$ be its 
topological dual. For each $A \in \Sigma$, let us
denote by $\chi_A$
the \emph{characteristic function} of $A$, defined by
\begin{eqnarray*}
\chi_A(t)=\left\{
\begin{array}{ll}
1 & \text{if  } t \in A, \\ \\
0 & \text{if  } t \in T \setminus A.
\end{array}\right. 
\end{eqnarray*}
A \emph{vector measure} is a
$\sigma$-additive set function $\mu: \Sigma \to X$. 
By the Orlicz-Pettis theorem (see also 
\cite[Corollary 1.4]{JD}), 
the $\sigma$-additivity of $\mu$ is equivalent to the
$\sigma$-additivity of 
the scalar-valued set function 
$x^{\prime} \mu : A \mapsto x^{\prime}(\mu(A))$ on 
$\Sigma $ for every $x^{\prime} \in X^{\prime}$.
For a literature on vector measures, see also 
\cite{CURBERA, JD, IRENE, KK, DR, OKADA, PANCHA} 
and the references therein. 

The \emph{variation} $|\mu|$ of $\mu$ is defined by setting
\begin{eqnarray*}\label{variation}
|\mu|(A)= \sup \left\{
\sum_{i=1}^r \|\mu(A_i)\|: A_i \in \Sigma,
i=1,2, \ldots, r; 
A_i \cap A_j = \emptyset \text{  for  }
i \neq j ; \, \bigcup_{i=1}^r A_i \subset A
\right\}.
\end{eqnarray*}
We define the \emph{semivariation} $\|\mu\|$ of $\mu$ by
\begin{eqnarray}\label{domination}
\|\mu\|(A)= \sup_{x^{\prime} \in X^{\prime},
\|x^{\prime}\| \leq 1} |x^{\prime} \mu|(A).
\end{eqnarray}
\begin{rem}\label{finitesemivariation}
\rm Observe that $\|\mu\|(A) < + \infty$ for all 
$A \in \Sigma$ (see also \cite[Corollary 1.19]{JD}, 
\cite[\S1]{OKADA}).
\end{rem}
The \emph{completion} of $\Sigma$ with respect to
$\|\mu\|$ is defined by 
\begin{eqnarray}\label{completion}
\widetilde{\Sigma}= \{A= B \cup N: B \in \Sigma, 
N \subset M \in \Sigma \text{  with  } \|\mu\|(M)=0\}.
\end{eqnarray}
A function $f:T \to \mathbb{R}$ is said to be 
\emph{$\mu$-measurable} if 
\begin{eqnarray*}\label{mumeasurability}
f^{-1}(B) \cap \{ t \in T: f(t) \neq 0 \}\in 
\widetilde{\Sigma} 
\end{eqnarray*}
for each Borel subset $B \subset \mathbb{R}$.

Observe that from (\ref{domination}) and (\ref{completion})
it follows that every $\mu$-measurable real-valued
function is also $x^{\prime} \mu$-measurable 
for every $x^{\prime} \in X^{\prime}$. 
Moreover, it is readily seen that every 
$\Sigma$-measurable real-valued function is also
$\mu$-measurable.   

We say that $\mu$ is \emph{$\Sigma$-separable} (or
\emph{separable}) if there is 
a countable family $\mathbb{B}=(B_k)_k$ in $\Sigma$ such that,
for each $A \in \Sigma$ and $\varepsilon > 0$, there is
$k_0 \in \mathbb{N}$ such that
\begin{eqnarray}\label{sep}
\|\mu\|(A \Delta B_{k_0})=
\sup_{x^{\prime} \in X^{\prime}, 
\|x^{\prime}\| \leq 1}
[\,|x^{\prime} \mu |(A \Delta B_{k_0}) ] \leq \varepsilon
\end{eqnarray}
(see also \cite{RICKER}). Such a family $\mathbb{B}$ is said 
to be \emph{$\mu$-dense}.

Observe that $\mu$ is $\Sigma$-separable if and only if 
$\Sigma$ is \emph{$\mu$-essentially countably generated},
namely there is a countably generated 
$\sigma$-algebra $\Sigma_0 \subset \Sigma$ such that
for each $A \in \Sigma$ there is $B \in \Sigma_0$ with
$\mu(A \Delta B)=0$. The separability of $\mu$ 
is satisfied, for instance, when 
$T$ is a separable metrizable space, $\Sigma$ is the
Borel $\sigma$-algebra of the Borel subsets of $T$, 
and $\mu$ is a Radon measure
(see also \cite[Theorem 4.13]{BREZIS}, 
\cite[Theorem 1.0]{DZAMONJA},
\cite[\S1.3 and \S2.6]{KK},
\cite[Propositions 1A and 3]{RICKER}).

From now on, we assume that $\mu$ is separable, and 
$\mathbb{B}=(B_k)_k$ is a $\mu$-dense family in $\Sigma$ with 
\begin{eqnarray}\label{semivariation}
\|\mu\|(B_k) \leq M= \|\mu\|(T)+1 \quad \text{for all   } k \in \mathbb{N}.
\end{eqnarray}

Now we recall the Henstock-Kurzweil (in short, 
$(HK)$)-integral 
for real-valued functions, defined on abstract sets,
with respect to (possibly infinite) non-negative measures.
For a related literature, see also 
\cite{BCSHENSTOCK, BRV, CDPMS0, CDPMS1, CSSISY, 
CSPANAM, CAO, DPHENSTOCK, DPM, AB, FM, GORDON, 
LEE, PFEFFER, RIECAN}
and the references therein.
When we deal with the $(HK)$-integral,
we assume that 
$T$ is a compact topological space and 
$\Sigma$ is the $\sigma$-algebra of all 
Borel subsets of $T$. We will not do these
assumptions to prove the results 
which do not involve the $(HK)$-integral.

Let $\nu:\Sigma \to \mathbb{R}
\cup \{+ \infty\}
$ be a $\sigma$-additive non-negative measure.  
A \emph{decomposition} of a set 
$A \in \Sigma$ is a finite collection
$\{(A_1, \xi_1)$, $(A_2,\xi_2),
\ldots, (A_N, \xi_N)\}$ such that $
A_j \in \Sigma$ and $\xi_j \in A_j$ 
for every $j \in \{1$, $2, \ldots, N\}$,
and $\nu(A_i \cap A_j)=0$ whenever $i \neq
j$. A decomposition of
subsets of $A \in \Sigma$ is called a \emph{partition} of $A$ when
$\displaystyle{\bigcup_{j=1}^N A_j=A}$.
A \emph{gauge} on a set $A \in \Sigma$ is a map $\delta$
assigning to each point $x \in A$ a neighborhood $\delta(x)$ of
$x$. If ${\mathcal D}=\{(A_1, \xi_1)$, $(A_2,\xi_2),
\ldots, (A_N, \xi_N)\}$ is a
decomposition of $A$ and $\delta$ is a gauge on $A$, then we say that
${\mathcal D}$ is \emph{$\delta$-fine} if $A_j \subset \delta(\xi_j)$
for any $j \in \{ 1$, $2, \ldots, N \}$.

An example is when 
$T_0$ is a 
locally compact and Hausdorff topological
space, and $T = T_0 \cup \{ x_0 \}$ is
the one-point compactification of $T_0$. In this case, we
will suppose that all involved functions $f$ 
vanish on $x_0$. For instance 
this is the case, when
$T_0 =\mathbb{R}^n$ is
endowed with the usual topology 
and $x_0$ is a point ``at the infinity", or when  
$T$ is the unbounded interval $[a, + \infty]=[a, +
\infty) \cup \{ + \infty \}$ of the extended real line,
considered as the one-point
compactification of the locally compact space $[a, + \infty)$. 
In this last case, the
base of open sets consists on 
the open subsets of $[a, + \infty)$ and
the sets of the type $(b, + \infty]$, where $ a
< b < + \infty$. Any gauge in $[a, + \infty]$ has the form
$\delta(x)=(x - d(x), x + d(x))$, if $x \in [a, + \infty] \cap
\mathbb{R}$, and 
$\delta(+\infty)=(b, + \infty] = (b, + \infty) \cup
\{ + \infty \}$, where $d$ denotes a positive real-valued function
defined on $[a, + \infty)$. 
Now we define the \emph{Riemann sums} by
$\displaystyle{
S(f,{\mathcal D})= \sum_{j=1}^N \, f(\xi_j) \nu (A_j)}$ 
if the sum exists in $\mathbb{R}$, with the convention 
$0 \cdot (+\infty)=0$.
Note that for any
gauge $\delta$ 
there exists at least one $\delta$-fine partition
${\mathcal D}$ such that $S(f,{\mathcal D})$ is well-defined.

A function $f:T \to \mathbb{R}$ is said to be 
\emph{Henstock-Kurzweil
integrable} (briefly, \emph{$(HK)$-integrable})
on a set $A\in \Sigma$ if there is an element
$I_A \in \mathbb{R}$ such
that for every $\varepsilon >0$ there is a gauge $\delta$ on
$A$ with
$| S(f, {\mathcal D}) - I_A| \leq \varepsilon$
whenever ${\mathcal D}$ is a $\delta$-fine partition of $A$
such that $S(f, {\mathcal D})$ exists in $\mathbb{R}$, 
and we write $$\displaystyle{(HK)
\int_A \, f \, d\nu}=I_A.$$ 
Observe that, if $A$, $B\in \Sigma$,
$B \subset A$ and $f:T \to \mathbb{R}$ is 
$(HK)$-integrable on $A$,  then $f$ is also
$(HK)$-integrable on $B$ and on $A \setminus B$, and 
\begin{eqnarray}\label{add}
(HK) \int_A f (t) \, d\nu= (HK) \int_B f (t) \, d\nu +
(HK) \int_{A \setminus B} f (t) \, d\nu  
\end{eqnarray}
(see also \cite[Propositions 5.14 and 5.15]{BRV},
\cite[Lemma 1.10 and Proposition 1.11]{RIECAN}).
From (\ref{add})
used with $A=T$ and  
$\chi_B f$ instead of $f$, it follows that, if $f$
is $(HK)$-integrable on $T$ and $B \in \Sigma$, then
\begin{eqnarray}\label{chiB}
(HK) \int_{T} \chi_B(t) f(t) \, d\nu = 
(HK) \int_B f(t) \, d\nu . 
\end{eqnarray}  

We say that
a $\Sigma$-measurable function $f:T \to \mathbb{R}$ 
is \emph{Kluv\'{a}nek-Lewis-Lebesgue $\mu$-integrable},
shortly \emph{$(KL)$ $\mu$-integrable} (resp. 
\emph{Kluv\'{a}nek-Lewis-Henstock-Kurzweil 
$\mu$-integrable},
shortly \emph{$(HKL)$ $\mu$-integrable}) if 
the following properties hold: 
\begin{eqnarray}\label{first}
f \text{   is   } 
|x^{\prime} \mu|\text{-Lebesgue (resp. } |x^{\prime} \mu| 
\text{-Henstock-Kurzweil) integrable  for each }  
x^{\prime} \in X^{\prime},
\end{eqnarray} and for every $A \in \Sigma$ there is
$x_A^{(L)}$ (resp. $x_A^{(HK)}) \in X$ with
\begin{eqnarray}\label{second} 
x^{\prime}(x_A^{(L)})=(L) \int_{A}f \, 
d |x^{\prime} \mu| \, 
\, \text{    (resp. } x^{\prime}(x_A^{(HK)})=
(HK) \int_{A}f \, d |x^{\prime}\mu|
\text{ )   for  all  }
x^{\prime} \in X^{\prime}, 
\end{eqnarray}
where the symbols $(L)$ and $(HK)$ in (\ref{second})
denote the usual Lebesgue (resp. Henstock-Kurzweil) integral
of a real-valued function with respect to an 
(extended) real-valued measure.
A $\Sigma$-measurable function $f:T \to \mathbb{R}$ 
is said to be \emph{weakly $(KL)$} (resp.
\emph{weakly $(HKL)$)} \emph{$\mu$-integrable} if it 
satisfies only condition (\ref{first}) 
(see also \cite{CG, CURBERA, PANCHA}). 
We recall the following facts about the $(KL)$-integral.
\begin{prop}\label{simple} \rm (see also 
\cite[Theorem 2.1.5 (i)]{PANCHA}) \em
If $s: T \to \mathbb{R}$, $\displaystyle{
s=\sum_{i=1}^r \alpha_i \chi_{A_i} }$ is $\Sigma$-simple,
with $\alpha_i \in \mathbb{R}$, $A_i \in \Sigma$, 
$i=1$, $2, \ldots, r$ and $A_i \cap A_j = \emptyset $
for $i \neq j$, then $s$ is $(KL)$ $\mu$-integrable on 
$T$, and $$(KL) \int_A s \, d\mu=
\sum_{i=1}^r \alpha_i \, \mu(A \cap A_i)
\text{   for   all  } A \in \Sigma .$$
\end{prop}
\begin{prop}\label{solid} \rm (see also 
\cite[Theorem 2.1.5 (vi)]{PANCHA} \em 
If $f:T \to \mathbb{R}$ is
$(KL)$ $\mu$-integrable on $T$ and
$A \in \Sigma$, then $\chi_A f$ is 
$(KL)$ $\mu$-integrable on $T$ and
\begin{eqnarray}\label{KLsubsets}
(KL) \int_A f \, d\mu=(KL) \int_{T} \chi_A  f \, d\mu.
\end{eqnarray}
\end{prop}
The space $L^1[\mu]$ (resp. $L_w^1[\mu]$) is
the space of all (equivalence classes of) $(KL)$ 
$\mu$-integrable functions (resp. weakly $(KL)$ 
$\mu$-integrable functions) up to the complement
of $\mu$-almost everywhere sets. For $p > 1$, the
space $L^p[\mu]$ (resp. $L_w^p[\mu]$) is
the space of all (equivalence classes of) 
$\Sigma$-measurable functions $f$ such that $|f|^p$ 
belongs to $L^1[\mu]$ (resp. $L_w^1[\mu]$).
The space $L^{\infty}[\mu]$ is the space of all
(equivalence classes of) 
$\mu$-essentially bounded functions. The norms 
are defined by
\begin{eqnarray}\label{Lpspace}
\left\{ \begin{array}{lll}
\|f\|_{L^p[\mu]} &= \|f\|_{L_w^p[\mu]} 
=\displaystyle{\sup\limits_{x^{\prime} \in X^{\prime}, 
\|x^{\prime}\|\leq 1} 
\left(
(L)\int_{T} 
|f(t)|^p \, d|x^{\prime}\mu| \right)^{1/p} }
& \text{if  } 1 \leq p < \infty, \\ \\ 
\|f\|_{L^{\infty}[\mu]} &= 
\sup\limits_{x^{\prime} \in X^{\prime}, 
\|x^{\prime}\| \leq 1} ( |x^{\prime}\mu|\textrm{-ess sup} |f|)
& 
\end{array} \right.
\end{eqnarray}
(see also \cite{BCSLp, AFP, S}).

If $f:T \to \mathbb{R}$ is a $(HKL)$-integrable
function, then the \emph{Alexiewicz norm} of $f$ is defined by 
\begin{eqnarray*}\label{alexiewicz}
||f||_{HKL}= \sup_{x^{\prime} \in X^{\prime}, 
\|x^{\prime}\|\leq 1} \left( \sup_{A \in \Sigma} 
\left| (HK) \int_A f(t) \, d|x^{\prime} \mu| \, \right| \, \right)
\end{eqnarray*} 
(see also \cite{ALEXIEWICZ, SWARTZ}).
Observe that, by arguing analogously as in
\cite[Theorem 9.5]{GORDON} and  
\cite[Example 3.1.1]{MCLEOD},  
for each $x^{\prime} \in X^{\prime}$
we get that $f=0$ 
$|x^{\prime} \mu|$-almost everywhere if and only if 
$\displaystyle{(HK) \int_A f(t) \, d|x^{\prime} \mu|}=0$
for every $A \in \Sigma$. 
Thus, it is not difficult to see that $\|\cdot \|_{HKL}$ is a norm.
In general, the space of the real-valued
Henstock-Kurzweil integrable functions
endowed with the Alexiewicz norm is not complete
(see also \cite[Example 7.1]{SWARTZ}).
\section{Construction of the Kuelbs-Steadman spaces
and main properties}
We begin with giving the following technical results, 
which will be useful later.
\begin{prop}\label{3.1}
Let $(a_k)_k$ and $(\eta_k)_k$ 
be two sequences of non-negative 
real numbers, such that $a=\sup\limits_k a_k < + \infty$, and
\begin{eqnarray}\label{series}
\sum\limits_{k=1}^{\infty} \eta_k=1, 
\end{eqnarray}
and $p >0$ be a fixed real number.
Then,
\begin{eqnarray}\label{used}
\displaystyle{ \left(\sum_{k=1}^{\infty} \eta_k \, a_k^p
\right)^{1/p}} \leq a.
\end{eqnarray}  
\end{prop}
\begin{proof} We have
$\eta_k \, a_k^p \leq a^p \, \eta_k$ for all $k \in \mathbb{N}$, and hence
$$\displaystyle{
\sum_{k=1}^{\infty} \eta_k \, a_k^p \leq a^p \, 
\sum_{k=1}^{\infty}\eta_k=a^p},
$$ getting (\ref{used}).
\end{proof}
\begin{prop}\label{minkowski}
Let $(b_k)_k$, $(c_k)_k$ be two sequences of real numbers,
$(\eta_k)_k$ be a sequence of positive 
real numbers, satisfying \rm (\ref{series})\em, and $p \geq 1$
be a fixed real number. Then, 
\begin{eqnarray}\label{mink1}
\left(\sum_{k=1}^{\infty}\eta_k |b_k + c_k|^p 
\right)^{1/p} \leq \left(\sum_{k=1}^{\infty} 
\eta_k (|b_k| + |c_k|)^p \right)^{1/p}
\leq \left( \sum_{k=1}^{\infty} \eta_k |b_k|^p \right)^{1/p}
+\left(\sum_{k=1}^{\infty} \eta_k |c_k|^p \right)^{1/p}.
\end{eqnarray} 
\end{prop} 
\begin{proof} It is a consequence of  
Minkowski's inequality (see also 
\cite[Theorem 2.11.24]{HLP}).
\end{proof}
Let $\mathbb{B}=(B_k)_k$ be as in
(\ref{semivariation}), set ${\mathcal E}_k= 
\chi_{B_k}$, $k \in \mathbb{N}.$\\

\noindent For $1 \leq p \leq \infty$, let us define a norm  
on $L^1[\mu]$ by setting 
\begin{eqnarray}\label{KSpspace}
\|f\|_{K{S^p}[\mu]} = \left\{
\begin{array}{ll} \displaystyle{
\sup_{x^{\prime} \in X^{\prime}, 
\|x^{\prime}\|\leq 1} 
\left\{  
\left[ \sum_{k=1}^{\infty} \eta_k \left| 
(L)\int_{T}{\mathcal E}_k(t)
f(t)d|x^{\prime}\mu| \right|^p\right]^{1/p}
\right\} }& \text{if  } 1 \leq p < \infty, \\ \\ \displaystyle{
\sup_{x^{\prime} \in X^{\prime}, 
\|x^{\prime}\| \leq 1} \left[ \sup_{k \in \mathbb{N}} \left| (L)
\int_{T} {\mathcal E}_k(t)f(t)
d|x^{\prime}\mu|
\right| \right] }& \text{if  } p=\infty.
\end{array}\right.
\end{eqnarray}
The following inequality holds.
\begin{prop}\label{preliminar}
For any  
$f \in L^1[\mu]$  and $p \geq 1$ it is
\begin{eqnarray}\label{pinfty}
\|f\|_{KS^p[\mu]} \leq \|f\|_{KS^{\infty}[\mu]} .
\end{eqnarray}
\end{prop}
\begin{proof}
By (\ref{used}) used with
\begin{eqnarray} \label{ak}
a_k = \left| (L) \int_{T}{\mathcal E}_k(t)
f(t) d|x^{\prime} \mu |(t)\right|, 
\end{eqnarray}
where 
$x^{\prime}$ is a fixed element of $X^{\prime}$ with
$\|x^{\prime}\|\leq 1$, we have
\begin{eqnarray}\label{index}
\left( \sum_{k=1}^{\infty} 
\eta_k \left| (L) \int_{T}{\mathcal E}_k(t)
f(t) d|x^{\prime} \mu |(t)\right|^p \right)^{1/p}
\leq \sup_{k \in \mathbb{N}} \left| (L)
\int_{T} {\mathcal E}_k(t)f(t)
d|x^{\prime}\mu|
\right|.
\end{eqnarray} Taking the supremum in (\ref{index})
as $x^{\prime} \in X^{\prime}$, 
$\|x^{\prime}\|\leq 1$, we obtain
\begin{eqnarray*}
\|f\|_{KS^p[\mu]}&=&
\sup_{x^{\prime} \in X^{\prime}, 
\|x^{\prime}\|\leq 1} 
\left\{ 
\left[ \sum_{k=1}^{\infty} \eta_k \left| 
(L)\int_{T}{\mathcal E}_k(t)
f(t)d|x^{\prime}\mu| \right|^p\right]^{1/p}
\right\}  \\ &\leq& 
\sup_{x^{\prime} \in X^{\prime}, 
\|x^{\prime}\| \leq 1} \left[ \sup_{k \in \mathbb{N}} \left| (L)
\int_{T} {\mathcal E}_k(t)f(t)
d|x^{\prime}\mu|
\right| \right] = \|f\|_{KS^{\infty}[\mu]},
\end{eqnarray*}
getting the assertion. \end{proof}

Now we prove that
\begin{thm}\label{norm}
The map $f \mapsto \|f\|_{KS^p[\mu]}$ defined 
in \rm (\ref{KSpspace}) \em is a norm.
\end{thm}
\begin{proof}
Observe that, by definition, 
$\|f\|_{KS^p[\mu]} \geq 0$ 
for every $f \in L^1[\mu]$. 
Let $f \in L^1[\mu]$ with
$\|f\|_{KS^p[\mu]}=0$. We prove that $f=0$
$\mu$-almost everywhere. It is enough
to take $1 \leq p < \infty$, since the case 
$p=\infty$ will follow from (\ref{pinfty}).
For $k \in \mathbb{N}$, let $a_k$ be as in (\ref{ak}).
As the $\eta_k$'s are strictly positive, from
\begin{eqnarray*}\label{zero}
\displaystyle{ \left(\sum_{k=1}^{\infty} \eta_k \, a_k^p
\right)^{1/p}} =0
\end{eqnarray*}  it follows that $a_k=0$ for every 
$k \in \mathbb{N}$. Hence,
\begin{eqnarray}\label{albaicin1}
\left|(L) \int_{T}{\mathcal E}_k(t)
f(t) \, d|x^{\prime} \mu |(t)\right|=0 \quad 
\text{for each   } k \in \mathbb{N} \text{  and   }
x^{\prime}\in X^{\prime} \text{  with  }\|x^{\prime}\| \leq 1.
\end{eqnarray}
Proceeding by contradiction, suppose that $f\neq 0$ 
$\mu $-almost everywhere. If
$E^+=f^{-1}(]0,+\infty[)$, $E^-=f^{-1}(]-\infty,0[)$, 
then $E^+$, $E^- \in \Sigma$, since $f$ is
$\Sigma$-measurable, and we have 
$\mu(E^+) \neq 0$ or $\mu(E^-) \neq 0$.
Suppose that $\mu(E^+) \neq 0$.
By the Hahn-Banach theorem, there is
$x^{\prime}_0 \in X^{\prime}$ with  
$\|x^{\prime}_0\| \leq 1$, $x^{\prime}_0\, \mu(E^+) 
\neq 0$, and hence $|x^{\prime}_0\, \mu(E^+)| 
> 0$. Moreover, if $f^*(t)=\min
\{f(t),1\}$, $t \in T$, then  
$E^+=\{t \in T: f^*(t)> 0\}$.  
For each $n \in \mathbb{N}$, set
\begin{eqnarray}
E_n^+=\left\{t \in T: 
\dfrac{1}{n+1} < f^*(t) \leq \dfrac{1}{n} \right\}.
\end{eqnarray} 
Since $\displaystyle{
E^+=\bigcup_{n=1}^{\infty} E^+_n}$ and $ x^{\prime}_0 \mu$ is 
$\sigma$-additive, there is $\overline{n} \in \mathbb{N}$
with $|x^{\prime}_0 \mu|(E^+_{\overline{n}}) >0$.
Put $\overline{B} =
E^+_{\overline{n}}$, and choose $\overline{\varepsilon}$ such that 
\begin{eqnarray}\label{epsilon}
0 < \overline{\varepsilon}
< \min\left\{ \dfrac{1}{\overline{n}+1} 
|x^{\prime}_0 \mu|(\overline{B}), 1 \right\}.
\end{eqnarray}
By the separability of $\mu$, in correspondence with
$\overline{\varepsilon}$ 
and $\overline{B}$ there is $B_{k_0} \in 
\mathbb{B}$ satisfying (\ref{sep}), that is 
\begin{eqnarray}\label{sep2}
\|\mu\|(\overline{B} \Delta B_{k_0})=
\sup_{x^{\prime} \in X^{\prime}, 
\|x^{\prime}\| \leq 1}
[\,|x^{\prime} \mu |(\overline{B} \Delta B_{k_0}) ] \leq \overline{\varepsilon}.
\end{eqnarray}
From (\ref{epsilon}) and (\ref{sep2}) we deduce
\begin{eqnarray*}\label{albaicinfinite}
\|\mu\|(B_{k_0}) \leq \|\mu\|(\overline{B}) + 
\|\mu\|(\overline{B} \Delta B_{k_0}) < \|\mu\|(T) + 1= M,
\end{eqnarray*} 
so that $B_{k_0} \in \mathbb{B}$, and 
\begin{eqnarray}\label{albaicinfin} \nonumber
\left|(L) \int_{T}\chi_{B_{k_0}}(t)
f(t) \, d|x_0^{\prime} \mu |(t)\right| &\geq&
(L) \int_{T}{\mathcal E}_{k_0}(t)
f(t) \,  d|x^{\prime}_0 \mu |(t) \\ &=& \nonumber
(L) \int_{B_{k_0}}
f(t) \,  d|x^{\prime}_0 \mu |(t) \geq 
(L) \int_{B_{k_0}}
f^*(t) \,  d|x^{\prime}_0 \mu |(t) \geq \\ &\geq&  
(L) \int_{\overline{B}}
f^*(t) \,  d|x^{\prime}_0 \mu |(t) -
(L) \int_{\overline{B}\Delta B_{k_0}}
f^*(t) \,  d|x^{\prime}_0 \mu |(t) \geq \\ &\geq&
\dfrac{1}{\overline{n}+1}  \nonumber
|x^{\prime}_0 \mu|(\overline{B}) - 
|x^{\prime} \mu |(\overline{B} \Delta B_{k_0})\geq \\ 
&\geq& \dfrac{1}{\overline{n}+1} 
|x^{\prime}_0 \mu|(\overline{B}) - 
\overline{\varepsilon} >0, \nonumber
\end{eqnarray}
which contradicts  
(\ref{albaicin1}). Therefore, $\mu(E^+)=0$.

Now, suppose that $\mu(E^-)\neq 0$. By proceeding 
analogously as in (\ref{albaicinfin}), replacing $f$
with $-f$ and $f^*$ with the function $f_*$
defined by $f_*(t)=
\min \{-f(t), 1 \}$, $t \in T$, we find an $x^{\prime}_1 \in
X^{\prime}$ with $\|x^{\prime}_1\| \leq 1$, an 
$\overline{\overline{n}} \in \mathbb{N}$,  
a $\overline{\overline{B}} \in \Sigma$, 
an $\overline{\overline{\varepsilon}} >0$ and 
a $B_{k_1} \in \mathbb{B}$ with
$\|\mu\|(B_{k_1})<M$, and 
$$\left|(L) \int_T \chi_{B_{k_1}}(t)
f(t) \, d|x_1^{\prime} \mu |(t)\right| \geq 
(L) \int_{B_{k_1}}
f_*(t) \,  d|x^{\prime}_1 \mu |(t)\geq
\dfrac{1}{\overline{\overline{n}}+1} 
|x^{\prime}_1 \mu|(\overline{\overline{B}}) - 
\overline{\overline{\varepsilon}} >0, 
$$ getting again a contradiction with 
(\ref{albaicin1}).
Thus, $\mu(E^-)= 0$, and $f=0$ 
almost everywhere. 

The triangular property of the norm can be  
deduced from Proposition \ref{minkowski} for
$1 \leq p < \infty$ and is not difficult to see for 
$p=\infty$,
and the other properties are easy to check.
\end{proof}
For $1 \leq p \leq \infty$, the \emph{Kuelbs-Steadman space} 
$K{S^p}[\mu]$ (resp. $K{S_w^p}[\mu]$) is
the completion of $L^1[\mu] $ (resp. $L_w^1[\mu]$) 
with respect to the norm 
defined in (\ref{KSpspace}) 
(see also \cite{BCSLp, AFP, GZ, KHM, KUELBS, S}). 
Observe that, to avoid ambiguity, we take the completion of
$L^1[\mu] $ rather than that of $L^p[\mu] $, but since the 
embeddings in Theorem \ref{embedding}
are continuous and dense, the two 
methods are substantially equivalent.

By proceeding similarly as in \cite[Theorem 3.26]{GZ},
we prove the following relations between 
the spaces $L^q[\mu]$ and $K{S^p}[\mu]$.
   \begin{thm}\label{embedding}
   For every $p$, $q$ with $1 \leq p \leq \infty$ and
$1 \leq q \leq \infty$, it is $L^q[\m] \subset K{S^p}[\mu]$ continuously and densely.
Moreover, the space of all $\Sigma$-simple functions is
dense in $K{S^p}[\mu]$.
   \end{thm}
   \begin{proof} 
We first consider the case $1 \leq p < \infty$.
Let $ f \in L^q[\mu] $, with $1 \leq q < \infty$, and $M$ be
as in (\ref{semivariation}). Note that $M^{\frac{q-1}{q}} \leq M$, 
since $M \geq 1$.
 As $|{\mathcal E}_k(t)|= {\mathcal E}_k(t)
\leq 1$ and $|{\mathcal E}_k(t)|^q\leq {\mathcal E}_k(t)$
for any $k \in \mathbb{N}$ and $t \in T$, taking into account (\ref{used}) 
and Jensen's inequality (see also \cite[Exercise 4.9]{BREZIS}), we deduce
\begin{eqnarray}\label{imbedding} \nonumber
\|f\|_{K{S^p}[\mu]} &=& \sup_{x^{\prime} \in X^{\prime},  
\|x^{\prime}\|\leq 1} 
\left\{ 
\left[ \sum_{k=1}^{\infty} \eta_k \left|
(L)\int_{T}{\mathcal E}_k(t)
f(t)d|x^{\prime}\mu| \right|^{\frac{pq}{q}}\right]^{1/p}
\right\}\\&\leq& \sup_{x^{\prime} \in X^{\prime}, 
\|x^{\prime}\|\leq 1} \left\{ \left[
\sum_{k=1}^{\infty} \eta_k \left( (|x^{\prime} \mu|(B_k))^{q-1} \, \cdot
(L)\int_{T}{\mathcal E}_k(t)|f(t)|^q 
d|x^{\prime}\mu| \right)^{p/q}\right]^{1/p} \right\} 
\\&\leq& M^{\frac{q-1}{q}} 
\sup_{x^{\prime} \in X^{\prime}, 
\|x^{\prime}\| \leq 1} \left[
\sup_{k \in \mathbb{N}} \left( (L)
\int_{T}{\mathcal E}_k(t)|f(t)|^q 
d|x^{\prime}\mu|
\right)^{1/q} \right]  \nonumber
\\ &\leq& M \, \sup_{x^{\prime} \in X^{\prime}, 
\|x^{\prime}\| \leq 1} \left[ \left( (L)
\int_{T}|f(t)|^q 
d|x^{\prime}\mu|
\right)^{1/q} \right] = M \, \,
\|f\|_{L^q[\m]}, \nonumber
\end{eqnarray} 
where $M$ is as in (\ref{semivariation}).
Now, let $1 \leq p < \infty$ and $q = \infty$. We have
\begin{eqnarray}\label{imbeddingqinfty} \nonumber
\|f\|_{K{S^p}[\mu]} &=& \sup_{x^{\prime} \in X^{\prime},  
\|x^{\prime}\|\leq 1} 
\left\{ 
\left[ \sum_{k=1}^{\infty} \eta_k \left|
(L)\int_{T}{\mathcal E}_k(t)
f(t)d|x^{\prime}\mu| \right|^p\right]^{1/p}
\right\}\\ &\leq& \sup_{x^{\prime} \in X^{\prime},  
\|x^{\prime}\|\leq 1} [ (|x^{\prime}\mu|(B_k))^p \cdot
\textrm{ess sup} |f|^p]^{1/p} \leq M \cdot
\|f\|_{L^{\infty}[\mu]}.
\end{eqnarray}
The proof of the
case $p=\infty$ is analogous to that of the case
$1 \leq p < \infty$.
Therefore, $ f \in K{S^p}[\mu]$, and the embeddings in
(\ref{imbedding}) and (\ref{imbeddingqinfty})
are continuous. 

Moreover, observe that every $\Sigma$-simple function  
belongs to $L^q[\mu]$ and the space of all 
$\Sigma$-simple functions is dense in $L^1[\mu]$
with respect to $\|\cdot\|_{L^1[\mu]}$ 
(see also \cite[Corollary 2.1.10]{PANCHA}).
Moreover, since $K{S^p}[\mu]$ is the completion of 
$L^1[\mu]$ with respect to the norm 
$\|\cdot\|_{K{S^p}[\mu]}$, the space $L^1[\mu]$ is dense 
in $K{S^p}[\mu]$ with respect to the 
norm $\|\cdot\|_{K{S^p}[\mu]}$ (see also 
\cite[\S4.4]{KA}).  

Choose arbitrarily $\varepsilon >0$ and $f \in 
K{S^p}[\mu]$. There is $g \in L^1[\mu]$ with 
$\|g-f\|_{K{S^p}[\mu]}\leq \dfrac{\varepsilon}{M+1}$.
Moreover, in correspondence with $\varepsilon$ and $g$
we find a $\Sigma$-simple function $s$, with  
$\|s-g\|_{L^1[\mu]}\leq 
\dfrac{\varepsilon}{M+1}$.
By (\ref{imbedding}) and (\ref{imbeddingqinfty}),
$\|\cdot\|_{K{S^p}[\mu]} \leq 
M \|\cdot\|_{L^1[\mu]}$, and hence
we obtain
\begin{eqnarray*}
\|s-f\|_{K{S^p}[\mu]} &\leq&   
\|s-g\|_{K{S^p}[\mu]} + \|g-f\|_{K{S^p}[\mu]}
 \\ &\leq& M
\|s-g\|_{L^1[\mu]} + \|g-f\|_{K{S^p}[\mu]} 
\leq  \dfrac{M \varepsilon}{M+1}
+ \dfrac{\varepsilon}{M+1}=\varepsilon ,
\end{eqnarray*}
getting the last part of the assertion. 
Thus, the embeddings in
(\ref{imbedding}) and (\ref{imbeddingqinfty}) are dense.
\end{proof} 
\begin{prop}\label{pinf}
${K{S^{\infty}}[\mu]}  \subset {K{S^p}[\mu]}$ for
every $p \geq 1$.
\end{prop}
\begin{proof} The assertion follows from (\ref{pinfty}),
since ${K{S^p}[\mu]}$ (resp.
${K{S^{\infty}}[\mu]})$ is the completion of 
$L^1[\mu]$ with respect to $\|f\|_{{K{S^p}[\mu]}}$
(resp. $\|f\|_{{K{S^{\infty}}[\mu]}}$).
\end{proof}
\begin{rem}\label{embeddingw}\rm
~ 
(a) Notice that, for $q \neq \infty$,
by Theorem \ref{embedding} and Proposition \ref{pinf}
hold also when 
$L^q[\mu]$ and $K{S^p}[\mu]$ are replaced by 
$L_w^q[\mu]$ and $K{S_w^p}[\mu]$, respectively. 

(b) If $f$ is $(HKL)$-integrable, 
then for each $x^{\prime} \in X^{\prime}$ and  
$k \in \mathbb{N}$, ${\mathcal E}_k f$ is 
both Henstock-Kurzweil and Lebesgue integrable
with respect to $|x^{\prime} \mu|$,
since $f$ is $\Sigma$-measurable, 
and the two integrals coincide, thanks to the 
$(HK)$-integrability of
the characteristic function $\chi_E$ for each $E \in \Sigma$
and the monotone convergence theorem (see also
\cite{BRV, RIECAN}). Thus, taking into account 
(\ref{index}), for every $p$ with $1 \leq p < \infty$ we have
 \begin{eqnarray*}
& & 
\sup_{x^{\prime} \in X^{\prime}, 
\|x^{\prime}\|\leq 1} 
\left[ \left(\sum_{k=1}^{\infty} \eta_k
\left|(L)\int_{T}{\mathcal E}_k(t)f(t)\,d|x^{\prime}\mu |   
\right|^p\right)^{1/p} \right]   \\ &\leq&
\sup_{x^{\prime} \in X^{\prime}, 
\|x^{\prime}\|\leq 1} \left( \sup_{k\in\mathbb{N}} 
\left| (L)\int_{T}
{\mathcal E}_k(t)f(t)\,d|x^{\prime}\mu |   
\right|\right) \\ &=& 
\sup_{x^{\prime} \in X^{\prime}, 
\|x^{\prime}\| \leq 1} \left( \sup_{k\in\mathbb{N}} 
\left| (HK)\int_{T}{\mathcal E}_k(t)
f(t)\,d|x^{\prime}\mu |   
\right|\right)
\\ &=& 
\sup_{x^{\prime} \in X^{\prime}, 
\|x^{\prime}\| \leq 1} \left( \sup_{k\in\mathbb{N}} 
\left| (HK)\int_{B_k} f(t)\,d|x^{\prime}\mu |   
\right|\right) \\&\leq& \sup_{x^{\prime} \in X^{\prime}, 
\|x^{\prime}\| \leq 1} \left( \sup_{A \in \Sigma} 
\left| (HK) \int_A f(t) \, d|x^{\prime} \mu| \, \right| \, \right)
= \|f\|_{HKL}. \qquad \qquad \qed
\end{eqnarray*}
\end{rem}
The next result deals with the separability of 
Kuelbs-Steadman spaces, which holds even for $p=\infty$,
differently from $L^p$ spaces.
\begin{prop}
For $1 \leq p \leq \infty$, the space
$K{S^p}[\mu]$ is separable. 
\end{prop}
\begin{proof} 
Observe that, by our assumptions, $\mu$ is separable, and 
this is equivalent to the
separability of the spaces $L^p[\mu]$ with $1\leq p<\infty$
(see also \cite[Proposition 2.3]{AFP},
\cite[Propositions 1A and 3]{RICKER}).

Now, let ${\mathcal H}=
\{h_n: n \in \mathbb{N} \}$ be a countable 
subset of $L^1$, dense in $L^1[\mu]$ with respect to
the norm $\|\cdot\|_{L^1[\mu]}$. By Theorem 
\ref{embedding}, ${\mathcal H} \subset K{S^p}[\mu]$.
We claim that ${\mathcal H}$ is dense in $K{S^p}[\mu]$.
Pick arbitrarily $\varepsilon >0$ and $f \in 
K{S^p}[\mu]$. There is $g \in L^1[\mu]$ with 
$\|g-f\|_{K{S^p}[\mu]}\leq 
\dfrac{\varepsilon}{M+1}$.
In correspondence with $\varepsilon$ and $g$
there exists  $n_0 \in \mathbb{N}$ such that  
$\|h_{n_0}-g\|_{L^1[\mu]}\leq 
\dfrac{\varepsilon}{M+1}$.
By (\ref{imbedding}),
$\|\cdot\|_{K{S^p}[\mu]} \leq M
\|\cdot\|_{L^1[\mu]}$,
and hence
\begin{eqnarray*}
\|h_{n_0}-f\|_{K{S^p}[\mu]} &\leq&   
\|h_{n_0}-g\|_{K{S^p}[\mu]} + \|g-f\|_{K{S^p}[\mu]}
 \\ &\leq& M \|h_{n_0}-g\|_{L^1[\mu]} + \|g-f\|_{K{S^p}[\mu]}
\leq \dfrac{M \varepsilon}{M+1}
+ \dfrac{\varepsilon}{M+1}=\varepsilon ,
\end{eqnarray*} 
getting the claim.
\end{proof}
Now we prove that $K{S^p}[\mu]$ spaces are Banach lattices
and K\"{o}the function spaces. First, we recall some properties
of such spaces (see also \cite{JT, MN}).

A partially ordered Banach space $X$ which is also a vector lattice is
a \emph{Banach lattice} if $\|x\| \leq \|y\|$ for every $x$,
$y \in X$ such that $|x|\leq |y|$.

A \emph{weak order unit} of $X$ is a positive element 
$e \in X$ such that, if $x \in X$ and
$x \wedge e =0$, then $x=0$.

Let $X$ be a Banach lattice and 
$\emptyset \neq A \subset B \subset X$.
We say that $A$ is \emph{solid} in $B$ if for each $x$, $y$ with
$x \in B$, $y \in A$ and $|x|\leq |y|$, it is $x\in A$.

Let $\lambda$ be an extended real-valued measure on  
$\Sigma$. A Banach space $X$ consisting of (classes of equivalence of)
$\lambda$-measurable functions is called a \emph{K\"{o}the
function space} with respect to $\lambda$
if, for every $g \in X$ and for each 
measurable function $f$ with $|f| \leq |g|$ 
$\lambda$-almost everywhere, it is $f \in X$ and $\|f\| \leq \|g\|$,
and $\chi_A \in X$
for every $A\in \Sigma$ with $\lambda(A)< + \infty$.
\begin{thm}
If $p \geq 1$, then $K{S^p}[\mu]$ is a Banach lattice 
with a weak order unit and 
a K\"{o}the function space with respect to a control
measure $\lambda$ of $\mu$.
\end{thm}
\begin{proof}
By the Rybakov theorem (see also \cite[Theorem IX.2.2]{JD}),
there is $x^{\prime}_0 \in X^{\prime}$ with  
$\|x^{\prime}_0\| \leq 1$, such that $\lambda= x_0^{\prime}
\mu$ is a control measure of $\mu$.
If $f$, $g \in K{S^p}[\mu]$,
$|f| \leq |g|$ $\lambda$-almost everywhere,
$k \in \mathbb{N}$ and
$x^{\prime} \in X^{\prime}$ with 
$\|x^{\prime}\| \leq 1$, then   
\begin{eqnarray}\label{monotonicity}
\left( (L)\int_{T}{\mathcal E}_k(t)
|f(t)|d|x^{\prime}\mu| \right)^p \leq
\left( (L)\int_{T}{\mathcal E}_k(t)
|g(t)|d|x^{\prime}\mu| \right)^p 
\end{eqnarray} 
(see also \cite[Proposition 5]{EP}), and hence
$\|f\|_{K{S^p}[\mu]} \leq \|g\|_{K{S^p}[\mu]}$.
By (\ref{monotonicity}), we can deduce that 
$K{S^p}[\mu]$ is a Banach lattice, because
$K{S^p}[\mu]$ is the completion of $L^1[\mu]$ with
respect to $\|\cdot\|_{K{S^p}[\mu]}$, $L^1[\mu]$
is a Banach lattice and the lattice operations 
are continuous with respect to norms
(see also \cite[Proposition 1.1.6 i)]{MN}).
Since $L^1[\mu]$ is solid with respect to the space 
of $\lambda$-measurable functions (see also 
\cite{PANCHA}) and the closure of every solid subset
of a Banach lattice is solid 
(see also \cite[Proposition 1.2.3 i)]{MN}), 
arguing similarly as in (\ref{monotonicity}) we obtain that,
if $f$ is $\lambda$-measurable, $g \in {K{S^p}[\mu]}$
and $|f| \leq |g|$ $\mu$-almost everywhere, then
$g \in {K{S^p}[\mu]}$. 

If $A \in \Sigma$, then $\lambda(A) < + \infty$ and 
$\chi_A \in {L^1[\mu]}$ (see also \cite[Proposition 5]{EP}), 
and hence 
$\chi_A \in {K{S^p}[\mu]}$. Therefore,
$K{S^p}[\mu]$ is  
a K\"{o}the function space. 

Finally, we prove that $\chi_T$ is a weak order unit of
$K{S^p}[\mu]$. First, note that 
$\chi_{T} \in
L^p[\mu]$, and hence $\chi_{T} \in 
K{S^p}[\mu]$. Let $f \in K{S^p}[\mu]$ be such 
that $f^*=f \wedge \chi_T = 0$
$\mu$-almost everywhere. We get
\begin{eqnarray}\label{ft0}
\{t \in T: f^*(t)=0 \}= \{t \in T: f(t)=0 \},
\end{eqnarray} and hence 
$f=0$ $\mu$-almost everywhere. 
This ends the proof.
\end{proof}
Note that, by the definition of the $(KL)$-integral, 
the norm defined in 
(\ref{KSpspace}) corresponds, in a certain sense, to the topology 
associated with norm convergence of the 
integrals (\emph{$\mu$-topology}, 
see also \cite[Theorem 2.2.2]{IRENE}). However, with this norm,
it is not natural to define an inner product in the space  
$KS^2$, since $m$ is vector-valued.

On the other hand, when $X^{\prime}$ is separable and 
$\{x^{\prime}_h$: $h \in \mathbb{N} \}$ is a countable dense subset of  
$X^{\prime}$, with $\|x^{\prime}_h\|\leq 1$ for every $h$,
it is possible to deal with the topology related 
to weak convergence of integrals (\emph{weak $\mu$-topology}, 
see also \cite[Proposition 2.1.1]{IRENE}), whose 
corresponding norm is given by 
\begin{eqnarray}\label{KSpspaceweaktopology}
\|f\|_{K{S^p}[w\tau\mu]} = \left\{
\begin{array}{ll} \displaystyle{
\left[ \sum_{h=1}^{\infty} \omega_h \left(
\sum_{k=1}^{\infty} \eta_k \left| 
(L)\int_{T}{\mathcal E}_k(t)
f(t)\,d|x^{\prime}_h\mu|  \right|^p\right) \right]^{1/p}
}& \text{if  } 1 \leq p < \infty, \\ \\ \displaystyle{
\sup_{h \in \mathbb{N}}
\left[ \sup_{k \in \mathbb{N}} \left| (L)
\int_{T} {\mathcal E}_k(t)f(t)
d|x^{\prime}_h\mu|
\right| \right] }& \text{if  } p=\infty,
\end{array}\right.
\end{eqnarray}
where ${\mathcal E}_k$, $k \in \mathbb{N}$,
is as in (\ref{KSpspace}), and
$(\eta_k)_k$, $(\omega_h)_h$ are two fixed sequences 
of strictly positive real numbers, such that 
$\displaystyle{\sum_{k=1}^{\infty}\eta_k=
\sum_{h=1}^{\infty}\omega_h=1}$.
Note that, in general, weak $\mu$-topology does not coincide with
$\mu$-topology, but there are some cases in which they are 
equal (see also \cite[Theorem 14]{EP}).
Analogously in Proposition \ref{preliminar}, 
it is possible to prove the following 
\begin{prop}\label{secondal}
For each     
$f \in L^1[\mu]$  and  $p \geq 1$, it is
\begin{eqnarray}\label{pinftybis}
\|f\|_{KS^p[w\tau\mu]} \leq \|f\|_{KS^{\infty}[w\tau\mu]} .
\end{eqnarray}
\end{prop}
Now we give the next fundamental result.
\begin{prop}\label{norm2}
The map $f \mapsto \|f\|_{KS^p[w\tau\mu]}$ defined 
in \rm (\ref{KSpspaceweaktopology}) \em is a norm.
\end{prop}
\begin{proof}
First of all note that 
$\|f\|_{KS^p[\mu]} \geq 0$ 
for any $f \in L^1[\mu]$. 
Let $f \in L^1[\mu]$ be such that
$\|f\|_{KS^p[\mu]}=0$. We prove that $f=0$
$\mu$-almost everywhere. It will be enough
to prove the assertion for
$1 \leq p < \infty$, since the case 
$p=\infty$ follows from (\ref{pinftybis}).
Arguing analogously as in (\ref{albaicin1}), we get
\begin{eqnarray}\label{albaicin2}
\left|(L) \int_{T}{\mathcal E}_k(t)
f(t) \, d|x^{\prime}_h \mu |(t)\right|=0 \quad 
\text{for every  } h, k \in \mathbb{N}.
\end{eqnarray}
By contradiction, suppose that $f\neq 0$ 
$\mu $-almost everywhere. If
$E^+=f^{-1}(]0,+\infty[)$, $E^-=f^{-1}(]-\infty,0[)$, 
then $E^+$, $E^- \in \Sigma$, since $f$ is
$\Sigma$-measurable, and we have 
$\mu(E^+) \neq 0$ or $\mu(E^-) \neq 0$.
Suppose that $\mu(E^+) \neq 0$.
By the Hahn-Banach theorem, there is
$x^{\prime}_0 \in X^{\prime}$ with  
$\|x^{\prime}_0\| \leq 1$, $x^{\prime}_0\, \mu(E^+) 
\neq 0$, and hence $|x^{\prime}_0\, \mu(E^+)| 
> 0$. Since the set $\{x^{\prime}_h$: $h \in \mathbb{N}\}$ 
is dense in $x^{\prime}$ with respect to the 
norm of $X^{\prime}$, there is a positive integer $h_0$ with
\begin{eqnarray}\label{albaicinsecondversion}
|x^{\prime}_{h_0}\, \mu(E^+)| > 0. 
\end{eqnarray}
Without loss of generality, we can assume 
$\|x^{\prime}_{h_0}\| \leq 1$. 
Now, the proof continues analogously as that of Theorem 
\ref{norm}, by replacing the linear continuous functional
$x^{\prime}_0$ in (\ref{albaicinfin})
with $x^{\prime}_{h_0}$ found in (\ref{albaicinsecondversion}),
by finding another element $x^{\prime}_{h_1} \in
X^{\prime}$ with
$|x^{\prime}_{h_1}\, \mu(E^-)| > 0$, and by arguing again as in
(\ref{albaicinfin}).

The triangular property of the norm is straightforward for
$p=\infty$, and for $1 \leq p < \infty$ is a consequence of 
the inequality
\begin{eqnarray}\label{mink2} \nonumber
\left[\sum_{h=1}^{\infty}\omega_h \left(
\sum_{k=1}^{\infty}\eta_k |b_{k,h} + c_{k,h}|^p 
\right) \right]^{1/p} &\leq& \left[\sum_{h=1}^{\infty}\omega_h
\left(\sum_{k=1}^{\infty} 
\eta_k (|b_{k,h}| + |c_{k,h}|)^p \right) \right]^{1/p}  \\ 
&\leq& \left[\sum_{h=1}^{\infty}\omega_h
\left( \sum_{k=1}^{\infty} 
\eta_k |b_{k,h}|^p \right)\right]^{1/p}
+\left[\sum_{h=1}^{\infty}\omega_h \left(\sum_{k=1}^{\infty} 
\eta_k |c_{k,h}|^p \right)\right]^{1/p}, 
\end{eqnarray} 
which holds whenever $(b_{k,h})_{k,h}$, $(c_{k,h})_{k,h}$ 
are two double sequences of real numbers, and 
$(\eta_k)_k$, $(\omega_h)_h$ are two sequences 
of positive real numbers, such that 
$\displaystyle{\sum_{h=1}^{\infty}\omega_h=
\sum_{k=1}^{\infty}\eta_k}=1$. The inequality in (\ref{mink2}), 
as that in (\ref{mink1}), follows from Minkowski's inequality.
The other properties are easy to check.
\end{proof}
Now, in correspondence with the norm defined 
in (\ref{KSpspaceweaktopology}), we define the following 
bilinear functional $\langle \cdot,\cdot \rangle : L^1[\mu] 
\times  L^1[\mu] \to \mathbb{R}$ by
\begin{eqnarray}\label{bilinear}
\langle f,g \rangle_{KS^2[w\tau\mu]}
=\sum_{h=1}^{\infty} 
\omega_h \left[\sum_{k=1}^{\infty} 
\tau_k \left((L) \int_{T}{\mathcal E}_k(t)
f(t) d|x^{\prime}_h \mu |(t)\right) \, 
\left((L)\int_{T}{\mathcal E}_k(s)
g(s) d|x^{\prime}_h \mu |(s)\right) \right].
\end{eqnarray}
Arguing similarly as in Theorem \ref{norm2}, it is possible to see
that the functional $\langle \cdot,\cdot \rangle_{KS^2[w\tau\mu]}$ in
(\ref{bilinear}) is an inner product, and  
$$\|\cdot\|_{K{S^2}[w\tau\mu]}=
(\langle \cdot,\cdot \rangle_{KS^2[w\tau\mu]})^{1/2}.$$ 
For $1 \leq p \leq \infty$, the \emph{Kuelbs-Steadman space} 
$K{S^p}[w\tau\mu]$ is
the completion of $L^1[\mu] $ 
with respect to the norm 
defined in (\ref{KSpspaceweaktopology}). 
Observe that, using Proposition \ref{3.1}, we can see that 
$$\|\cdot\|_{K{S^p}[w\tau\mu]} \leq \|\cdot\|_{K{S^p}[\mu]}
\text{   and   }  
\|\cdot\|_{K{S^p}[w\tau\mu]} \leq \|\cdot\|_{HKL}
 \text{   for  } 1 \leq p \leq \infty.$$
As in Theorem \ref{embedding},
it is possible to prove the following  
   \begin{thm}\label{embedding2}
   For each $p$, $q$ with $1 \leq p \leq \infty$ and
$1 \leq q \leq \infty$, it is 
$L^q[\mu] \subset K{S^p}[w\tau\mu]$ continuously and densely,
and the space of all $\Sigma$-simple functions is
dense in $K{S^p}[w\tau\mu]$.
Moreover, ${K{S^p}[w\tau\mu]}$ is a 
separable Banach lattice 
with a weak order unit and 
a K\"{o}the function space with respect to a control
measure $\lambda$ of $\mu$.
\end{thm}
Since $({K{S^2}[w\tau\mu]}, \langle \cdot,\cdot \rangle_{KS^2[w\tau\mu]})$ is a separable Hilbert space, by 
applying \cite[Theorems 5.15 and 8.7]{GZ}, 
it is possible to consider operators like, for instance,
convolution and Fourier transform, and to extend the theory 
there studied to the context of
vector-valued measures (see also \cite{OP}, \cite[Remark 5.16]{GZ}).
\section{Conclusions}
We have introduced Kuelbs-Steadman spaces 
related to integration for scalar-valued functions with 
respect to a $\sigma$-additive measure $\mu$, 
taking values in
a Banach space $X$. We have endowed them
with the structure of Banach space, both in connection with 
norm convergence of integrals and in connection with weak
convergence of integrals ($KS^p[\mu]$ and 
$KS^p[w\tau\mu]$, respectively). A fundamental role is 
played by the separability of $\mu$. 
We have proved that these spaces are separable Banach 
lattices and K\"{o}the function spaces, and  
can be embedded 
continuously and densely in the spaces 
$L^q[\mu]$. When $X^{\prime}$ is 
separable, we have endowed 
$KS^2[w\tau\mu]$ with an inner product. 
In this case, $KS^2[w\tau\mu]$ is a separable 
Hilbert space, and hence it is possible to 
deal with operators like
convolution and Fourier transform, and to extend 
to Banach space-valued measures the
theory investigated in \cite{GILLSURVEY, GZ, KHM}.  
\section*{Acknowledgments}
This research was partially supported by 
University of Perugia, 
the  G. N. A. M. P. A. (the Italian National Group of
Mathematical Analysis, Probability and Applications), 
and by the projects
``Ricerca di Base 2017'' (Metodi di Teoria
dell'Approssimazione e di Analisi Reale per problemi di
approssimazione ed applicazioni), 
``Ricerca di Base 2018'' (Metodi di Teoria
dell'Approssimazione, Analisi Reale, Analisi Nonlineare
e loro applicazioni) and ``Ricerca di Base 2019'' (Metodi di
approssimazione, misure, analisi funzionale, statistica e 
applicazioni alla ricostruzione di immagini e documenti).
\section*{Conflict of interest}
The authors declare that they have no conflict of interest.
\section*{Author contributions:} Conceptualization, B.H. and H.K.; methodology, A.B., B.H., and H.K.; validation,
A.B.; formal analysis, A.B.; investigation, A.B., B.H., and H.K.; resources, A.B., B.H., and H.K.; data curation,
A.B.; writing, original draft preparation, B.H. and H.K.; writing, review and editing, A.B.; visualization, A.B.;
supervision, A.B.; project administration, A.B.; funding acquisition, A.B. All authors have read and agreed to the
published version of the manuscript.
\bibliographystyle{amsalpha}

\end{document}